\documentclass[12pt,reqno]{amsart}
\usepackage[pagebackref,colorlinks=true]{hyperref}
\usepackage{verbatim}
\usepackage{eucal,url,amssymb,stmaryrd,enumerate,amscd,}
\usepackage{amsfonts,amsmath,amsthm,amssymb,amscd,enumerate,eucal,url,stmaryrd}
\usepackage{mathtools}
\usepackage[margin=1in]{geometry}
\usepackage{xcolor}
\usepackage{tikz}
\usetikzlibrary{arrows.meta,positioning}

\usepackage{mathrsfs}

\numberwithin{equation}{section}

\newcommand{\R}{\mathbb{R}}

\newcommand{\Rp}{\R_{>0}}
\newcommand{\Rpp}{\R_{\geq 0}}
\newcommand{\e}{\mathrm{e}}

\makeatother

\newtheorem{Theorem}{Theorem}[section]
\newtheorem{Lemma}{Lemma}[section]
\newtheorem{Proposition}{Proposition}[section]
\newtheorem{Corollary}{Corollary}[section]
\newtheorem{Definition}{Definition}[section]
\newtheorem{Remark}{Remark}[section]

\usepackage{amsthm}

\newtheoremstyle{axiomstyle}
  {}{}                 
  {\itshape}           
  {}                   
  {\bfseries}          
  {.}                  
  {0.5em}              
  {\thmname{#1}~\thmnumber{#2}\thmnote{ (#3)}} 
  
\theoremstyle{axiomstyle}

\begin{document}

\newcommand{\add}[1]{{\color{blue}#1}}
\newcommand{\del}[1]{{\color{red}#1}}
\newenvironment{added}{\begingroup\color{blue}}{\endgroup}
\newenvironment{deleted}{%
  \begingroup\color{red}%
  \renewcommand{\cite}[2][]{\relax}%
  \renewcommand{\label}[1]{\relax}%
  \let\ref\relax
  \let\eqref\relax
  \let\pageref\relax
  \renewcommand{\section}[1]{\par\medskip\noindent{\bfseries [deleted section] ##1}\par}%
  \renewcommand{\subsection}[1]{\par\medskip\noindent{\bfseries [deleted subsection] ##1}\par}%
  \renewcommand{\subsubsection}[1]{\par\medskip\noindent{\bfseries [deleted subsubsection] ##1}\par}%
  \renewcommand{\paragraph}[1]{\par\noindent{\bfseries [deleted paragraph] ##1}\par}%
  \renewcommand{\subparagraph}[1]{\par\noindent{\bfseries [deleted subparagraph] ##1}\par}%
  \renewenvironment{equation}{\[\ignorespaces}{\]\ignorespacesafterend}%
  \renewenvironment{theorem}[1][]{\par\medskip\noindent\textbf{[deleted theorem] }\ignorespaces}{\par\medskip}%
  \renewenvironment{lemma}[1][]{\par\medskip\noindent\textbf{[deleted lemma] }\ignorespaces}{\par\medskip}%
  \renewenvironment{proposition}[1][]{\par\medskip\noindent\textbf{[deleted proposition] }\ignorespaces}{\par\medskip}%
  \renewenvironment{corollary}[1][]{\par\medskip\noindent\textbf{[deleted corollary] }\ignorespaces}{\par\medskip}%
  \renewenvironment{definition}[1][]{\par\medskip\noindent\textbf{[deleted definition] }\ignorespaces}{\par\medskip}%
  \renewenvironment{remark}[1][]{\par\medskip\noindent\textbf{[deleted remark] }\ignorespaces}{\par\medskip}%
  \renewenvironment{example}[1][]{\par\medskip\noindent\textbf{[deleted example] }\ignorespaces}{\par\medskip}%
  \renewenvironment{notation}[1][]{\par\medskip\noindent\textbf{[deleted notation] }\ignorespaces}{\par\medskip}%
  \renewenvironment{convention}[1][]{\par\medskip\noindent\textbf{[deleted convention] }\ignorespaces}{\par\medskip}%
  \renewenvironment{axiom}[1][]{\par\medskip\noindent\textbf{[deleted axiom] }\ignorespaces}{\par\medskip}%
}{\endgroup}

\begin{abstract}

 We study a rigidity problem for functions \(F:\R_{>0}\to\R_{\ge 0}\) that penalize deviation of a positive ratio from equilibrium \(x=1\).
Assuming 
(i) a d’Alembert-type composition law on \(\R_{>0}\), and (ii) a single quadratic calibration at the identity (in logarithmic coordinates), we prove that  \(F\) is uniquely determined. The composition law implies the normalization 
$F(1)=0.$
The unique solution is called the canonical reciprocal cost, namely the difference between the arithmetic and geometric means of \(x\) and its reciprocal. Our proof uses the logarithmic coordinates \(H(t)=F(e^t)+1\), where the composition law becomes d’Alembert’s functional equation on \(\R\).
The calibration provides the minimal regularity needed to invoke the classical classification of continuous solutions and fixes the remaining scaling freedom, selecting the hyperbolic-cosine branch.
We also establish necessity of each assumption: without calibration the composition law admits a continuous one-parameter family, without the composition law the calibration does not determine the global form, and without regularity the composition law admits pathological non-measurable solutions. Finally, we establish a stability estimate for approximate
solutions under bounded defect and characterize some properties of the canonical cost.


\end{abstract}

\title[Uniqueness of the Canonical Reciprocal Cost]{Uniqueness of the Canonical Reciprocal Cost}
\date{\today}

\author{Jonathan Washburn}
\address[Jonathan Washburn]{Recognition Physics Institute Austin, Texas, USA}
\email{jon@recognitionphysics.org}
\author{Milan Zlatanovi\'c}
\address[Milan Zlatanovi\'c]{Department of Mathematics, Faculty of Science and Mathematics, University of Ni\v s, Vi\v segradska 33, 18000 Ni\v s, Serbia}
\email{zlatmilan@yahoo.com}


\maketitle 

\setcounter{tocdepth}{3}


\newcommand{\config}{\mathcal{C}}
\newcommand{\configR}{\mathcal{C}_R}

\section{Introduction}

\noindent
In the paper, we consider a function \(F:\R_{>0}\to\R_{\geq0}\)
with a minimum at \(x=1\).
Reciprocity \(F(x)=F(x^{-1})\) is natural for ratio costs, as it penalizes
reciprocal deviations equally.

\smallskip
Our central assumption is a d’Alembert-type {composition law} on \(\R_{>0}\) (simply called the composition law),
\[
F(xy)+F\Big(\dfrac{x}{y}\Big)=2F(x)F(y)+2F(x)+2F(y).
\]
Under this equation, reciprocity \(F(x)=F(x^{-1})\)
and the normalization \(F(1)=0\) follow.

Passing to logarithmic coordinates \(H(t):=F(e^t)+1\),
the composition law becomes d’Alembert’s functional equation
\[
H(t+u)+H(t-u)=2H(t)H(u).
\]
The theory of d’Alembert’s equation is classical, continuous solutions on \(\R\) are classified \cite{Aczel,Kuczma,Papp,Kannappan2,StetkaerBook,AczelDhombres,Czerwik}, discrete restrictions produce Chebyshev structure \cite{Davison}, and additional regularity assumptions yield strong rigidity \cite{Akkouchi,EbanksStetkaer}.
We use a single local calibration at equilibrium
\[
\lim_{t\to 0}\frac{2F(e^t)}{t^2}=1,
\]
which supplies the needed regularity and fixes the remaining scaling freedom.

{Assuming a d’Alembert-type composition law on \(\R_{>0}\)
and a local quadratic calibration at the identity (in logarithmic coordinates),
we prove that \(F\) is uniquely determined and equals the canonical reciprocal cost
\[
J(x)=\frac12(x+x^{-1})-1.
\]}

\smallskip

The paper is organized as follows.
Section~\ref{sec:prelim} states the main theorem, introduces the logarithmic coordinates,
and recalls the relevant facts about d’Alembert’s equation.
Section~\ref{sec:results} proves the rigidity theorem and demonstrates the necessity
of each assumption via explicit counterexamples and pathological solutions.
Section~\ref{sec:stability} establishes a quantitative consistency estimate
for approximate solutions
to the d’Alembert equation with bounded defect. 
Section~\ref{sec:properties} presents further structural properties of \(J\),
including its Bregman divergence, induced metric, and Chebyshev structure.

\section{Definitions and basic properties}\label{sec:prelim}

In this section we give the basic definitions and notation, state the main rigidity theorem, and rewrite the problem in logarithmic coordinates.
The proof of the rigidity theorem is given in Section~\ref{sec:results}.

\subsection{Logarithmic coordinates}\label{subsec:log-coords-main}
 
For any \(F:\Rp\to\R\), we introduce logarithmic coordinates \(t=\ln x\) and define
the associated functions
\begin{equation}
\label{GH-main}
G(t):=F(\e^t),\qquad H(t):=G(t)+1=F(\e^t)+1,\qquad t\in\R.
\end{equation}
If \(F(1)=0\), then \(G(0)=0\) and \(H(0)=1\).

\begin{Lemma}\label{lem:forced-normalization}
Let $F:\Rp\to\R$ satisfy the composition law
\[
F(xy)+F\!\left(\frac{x}{y}\right)=2F(x)F(y)+2F(x)+2F(y).
\]
Then either $F\equiv -1$, or $F(1)=0$.
\end{Lemma}
\begin{proof}
Set $y=1$. Then
\[
2F(x)=2F(x)F(1)+2F(x)+2F(1),
\]
hence $0=2F(1)(F(x)+1)$ for all $x>0$.
Thus either $F(1)=0$ or $F(x)=-1$ for all $x>0$.\end{proof}

\subsection{The canonical reciprocal cost}
\begin{Definition} 
The function \(J:\Rp\to\R\) defined by
\begin{equation}\label{canonreccost}
J(x):=\frac{x+x^{-1}}{2}-1
\end{equation}
is called {\rm the canonical reciprocal cost} function.
\end{Definition}

The function \(J\) satisfies:
\begin{itemize}
    \item[(i)]  \emph{Normalization:} \(J(1)=0\);
    \item[(ii)]  \emph{Reciprocity:} \(J(x)=J(x^{-1})\);
    \item[(iii)] \emph{Nonnegativity:} for all $x\in \Rp$, \(J(x)=\frac{(x-1)^2}{2x}\ge 0\).
\end{itemize}
In logarithmic coordinates \(t=\ln x\), one has \(J(\e^t)=\cosh(t)-1\).

\subsection{Calibration}
\begin{Definition}\label{def:calibration-early}
Let \(F:\Rp\to\R\). 
The \emph{log-curvature} of \(F\), denoted \(\kappa(F)\), is defined as
\[
\kappa(F) \;:=\; \lim_{t\to 0}\frac{2\,F(\e^t)}{t^2},
\]
provided this limit exists.
\end{Definition}

When the limit exists, $\kappa(F)$ represents the quadratic coefficient in the local behavior of $F(e^t)$ at $t=0$. 
This definition does not assume a priori that $F$ is $C^2$; the existence of the limit provides the necessary quadratic behavior near $t=0$.

By the change of variables $x=e^t$, the limit exists if and only if
\[
\lim_{x\to 1}\frac{2\,F(x)}{(\ln x)^2}
\]
exists, and in that case the two limits coincide.

{ 
The condition $\kappa(F)=1$ excludes the constant solution $F\equiv -1$ and forces $F(1)=0$.}

\subsection{Main result}
We now state the main theorem, which establishes the uniqueness of the canonical reciprocal cost under the composition law and a single quadratic calibration.

\begin{Theorem}\label{thm:main}
Let \(F:\Rp\to\R\). Assume that $F$ satisfies:
\begin{enumerate}
  \item[(i)] {Composition law on \(\Rp\):} for all \(x,y>0\),
  \begin{equation}\label{eq1}
  F(xy)+F\Big(\dfrac xy\Big)=2\,F(x)\,F(y)+2\,F(x)+2\,F(y);
  \end{equation}
  \item[(ii)] \emph{Unit log-curvature} (Definition~\ref{def:calibration-early}): \(\kappa(F)=1\), i.e.
  \begin{equation}\label{behavior0}
  \lim_{t\to 0}\frac{2\,F(\e^t)}{t^2}=1.
  \end{equation}
\end{enumerate}
Then for all \(x>0\),
\[
F(x)=\frac{x+x^{-1}}{2}-1 \;=\; J(x).
\]
\end{Theorem}

\begin{Definition}\label{def:recip-norm}
A function \(F:\Rp\to\R\) is called a \emph{reciprocal cost} if
\[
F(x)=F(x^{-1})\qquad\text{for all }x>0.
\]
It is called \emph{normalized} if $F(1)=0$.
\end{Definition}

\begin{Remark}
Reciprocity is a natural symmetry for ratio costs. 
In Proposition~\ref{prop:reciprocity-forced}, we show that reciprocity is implied by the composition law. 
\end{Remark}


\begin{Lemma}\label{lem:recip-even}
If \(F\) is reciprocal, then \(G\) and \(H\) given by \eqref{GH-main} are even.
\end{Lemma}
\begin{proof}
Since \(\e^{-t}=(\e^t)^{-1}\) and \(F(x)=F(x^{-1})\), we have
\[
G(-t)=F(\e^{-t})=F\bigl((\e^t)^{-1}\bigr)=F(\e^t)=G(t).
\]
Further, for $H(-t)$, we have \[H(-t)=G(-t)+1=G(t)+1=H(t).\]
\end{proof}

\subsection{The d'Alembert functional equation}

The key structural identity in this paper is d’Ale\-mbert’s functional equation (also called the cosine equation), which is equivalent to our composition law on \(\Rp\) (Lemma~\ref{lem:equiv-rp-dalembert}).

\begin{Definition} \label{def:dalembert}
A function \(H:\R\to\R\) is said to satisfy the d'Alembert functional equation if,
for all \(t,u\in\R\),
\begin{equation}\label{dal}
H(t+u)+H(t-u)=2\,H(t)\,H(u).
\end{equation}
\end{Definition}

The equation goes back to d’Alembert \cite{dAlembert1769} and was further studied by Poisson \cite{Poisson1804} and Picard \cite{Picard1922}, and many others.
The equation plays an important role in determining the sum of two vectors in
various Euclidean and non-Euclidean geometries.
 
\smallskip

{ 
Substituting $t=u=0$ in \eqref{dal} gives $H(0)\in\{0,1\}$.
If $H(0)=0$, then \eqref{dal} forces $H(x)=0$ for all $x>0$, which implies
$F(x)=-1$ for all $x>0$. This contradicts the assumption that
$F$ is nonconstant. Hence $H(0)=1$.
}

\begin{Theorem}[\cite{Papp}]
If \(H:\R\to\R\) is continuous and satisfies \eqref{dal}, then either \(H\equiv 0\), or \(H\equiv 1\), or
\[
H(t)=\cos(kt),\qquad \text{or}\qquad H(t)=\cosh(kt),
\]
for some real constant \(k\).
\end{Theorem}


More generally, the following holds.
\begin{Theorem}[\cite{Kannappan}]\label{thm:kannappan-general} The general complex-valued solutions of d'Alembert's functional equation \eqref{dal}
on the cartesian square of an abelian group $G$ are given by
\[
H(x)=\frac{h(x)+h(-x)}{2},
\]
where \(h : G \to \mathbb{C}\) satisfies 
\(
h(x+y)=h(x)h(y), \; x,y\in G.
\)
\end{Theorem}

In this paper we consider d’Alembert’s functional equation on the group
\(G=(\mathbb{R},+)\) for functions \(H:\mathbb{R}\to\mathbb{R}\).

If one assumes continuity of $H$ (as we will later obtain from the curvature calibration), then the multiplicative solutions \(h\) have the classical exponential form \(h(t)=e^{ct}\) with \(c\in \mathbb{C}\), and hence
\[
H(t)=\frac{e^{ct}+e^{-ct}}{2}.
\]
For real-valued \(H\), this reduces to the standard cosine/hyperbolic-cosine classification (cf. \cite{Aczel,Kuczma,Papp}).


\begin{Lemma} \label{lem:dalembert-even}
If \(H\) satisfies Definition~\ref{def:dalembert}, then \(H\) is even.
\end{Lemma}
\begin{proof}
Fix \(u\in\R\) and apply the d'Alembert equation (\ref{dal}) with \(t=0\):
\[
H(u)+H(-u)=2\,H(0)\,H(u)=2H(u),
\]
so \(H(-u)=H(u)\).
\end{proof}

\begin{Lemma}\label{lem:dalembert-product}
If \(H\) satisfies Definition~\ref{def:dalembert}, then for all \(t,u\in\R\),
\[
H(t+u)\,H(t-u)=H(t)^2+H(u)^2-1.
\]
\end{Lemma}
\begin{proof}
Apply (\ref{dal}) with \(a=t+u\) and \(b=t-u\):
\[
H((t+u)+(t-u))+H((t+u)-(t-u))=2H(t+u)H(t-u),
\]
so \[H(2t)+H(2u)=2H(t+u)H(t-u).\] Using that
\[
H(2t)=2H(t)^2-1
\]
(obtained from (\ref{dal}) with \((t,t)\) and \(H(0)=1\)), and similarly for \(u\), yields the claim.
\end{proof}

\begin{Lemma}\label{lem:dalembert-diff-square}
If \(H\) satisfies Definition~\ref{def:dalembert}, then for all \(t,u\in\R\),
\[
\bigl(H(t+u)-H(t-u)\bigr)^2=4\,(H(t)^2-1)\,(H(u)^2-1).
\]
\end{Lemma}
\begin{proof}
Let \(A:=H(t+u)\) and \(B:=H(t-u)\). Then \(A+B=2H(t)H(u)\) by Definition \ref{def:dalembert}, and \(AB=H(t)^2+H(u)^2-1\)
by Lemma~\ref{lem:dalembert-product}. Hence
\begin{align*}
        (A-B)^2&=(A+B)^2-4AB\\
&=4H(t)^2H(u)^2-4(H(t)^2+H(u)^2-1)\\
&=4(H(t)^2-1)(H(u)^2-1).
\end{align*}
\end{proof}

\begin{Lemma}\label{lem:dalembert-continuity}
{If \(H\) satisfies Definition~\ref{def:dalembert} and if 
\(\lim_{t\to 0}2(H(t)-1)/t^2\) exists, then \(H\) is continuous on \(\R\).}
\end{Lemma}
\begin{proof}
The limit assumption implies that \(\lim_{t\to 0} H(t)=1\).
Since \(H(0)=1\), it follows that \(H\) is continuous at \(0\).

Fix \(t\in\R\). For \(u\to 0\), the equation (\ref{dal}) gives
\[
\lim_{u\to 0}\bigl(H(t+u)+H(t-u)\bigr)
=2H(t)\lim_{u\to 0}H(u)
=2H(t).
\]
By Lemma~\ref{lem:dalembert-diff-square}, we have
\[
\lim_{u\to 0}\bigl(H(t+u)-H(t-u)\bigr)^2
=4\bigl(H(t)^2-1\bigr)\lim_{u\to 0}\bigl(H(u)^2-1\bigr)=0,
\]
hence
\[
\lim_{u\to 0}\bigl(H(t+u)-H(t-u)\bigr)=0.
\]
Moreover,
\[
H(t+u)
=\frac{H(t+u)+H(t-u)}{2}
+\frac{H(t+u)-H(t-u)}{2}.
\]
Taking limits as \(u\to 0\) and using
\[
\lim_{u\to 0}\bigl(H(t+u)+H(t-u)\bigr)=2H(t),
\qquad
\lim_{u\to 0}\bigl(H(t+u)-H(t-u)\bigr)=0,
\]
we obtain
\[
\lim_{u\to 0} H(t+u)=H(t).
\]
Similarly, \(\lim_{u\to 0} H(t-u)=H(t)\).
Therefore, \(H\) is continuous at every \(t\in\R\).
\end{proof}

Let \(H\) satisfy \eqref{dal} and set \(G:=H-1\).
In this case, a direct calculation shows that \(G\) satisfies
\[
G(t+u)+G(t-u)=2\,G(t)\,G(u)+2\,G(t)+2\,G(u).
\]
For example (see \cite{Papp}), the function
\[
H(t)=J(\e^t)+1=\cosh(t)
\]
satisfies the d'Alembert equation~(\ref{dal}), and the corresponding
\(G(t)=\cosh(t)-1\) satisfies the identity above.

\subsection{Composition law on \texorpdfstring{$\Rp$}{R\_>0}}\label{sec:rp-law}

Although our main theorem is stated on \(\Rp\), the composition law is most transparent after passing to logarithmic coordinates, where it reduces exactly to d'Alembert's equation for \(H(t)=F(e^t)+1\).

\begin{Definition}\label{def:rp-law}
A function \(F:\Rp\to\R\) satisfies the \emph{composition law on \(\Rp\)} if for all
\(x,y>0\),
\begin{equation}\label{comp}
F(xy)+F\Big(\dfrac xy\Big)=2\,F(x)\,F(y)+2\,F(x)+2\,F(y).
    \end{equation}
\end{Definition}

\begin{Lemma}\label{lem:equiv-rp-dalembert}
Let \(F:\Rp\to\R\), and  \(H:\R\to\R\) such that \(H(t)=F(\e^t)+1\).
Then \(F\) satisfies Definition~\ref{def:rp-law} if and only if \(H\) satisfies the d'Alembert
equation (\ref{dal}).
\end{Lemma}
\begin{proof}
Assume \(F\) satisfies Definition~\ref{def:rp-law}. Let \(t,u\in\R\) and set \(x=\e^t\), \(y=\e^u\). Then
\begin{align*}
H(t+u)+H(t-u)
&=\bigl(F(\e^{t+u})+1\bigr)+\bigl(F(\e^{t-u})+1\bigr)\\
&=\bigl(F(xy)+F\Big(\dfrac xy\Big)\bigr)+2\\
&=\bigl(2F(x)F(y)+2F(x)+2F(y)\bigr)+2\\
&=2\bigl(F(x)+1\bigr)\bigl(F(y)+1\bigr)
=2H(t)H(u),
\end{align*}
so \(H\) satisfies (\ref{dal}).

Conversely, if \(H\) satisfies (\ref{dal}), by reverse calculation with \(x=\e^t\), \(y=\e^u\), we obtain
that $F$ satisfies Definition~\ref{def:rp-law}.
\end{proof}

\section{Main results}\label{sec:results}

This section contains the rigidity argument for Theorem~\ref{thm:main}.
Using logarithmic coordinates \(H(t)=F(e^t)+1\), the composition law on \(\Rp\) becomes d’Alembert’s equation~\eqref{dal}.
We verify the assumptions for the canonical cost \(J\), classify the solutions, and obtain uniqueness.
We also briefly discuss which assumptions are essential.

\begin{Lemma}\label{lem:J-meets}
Let
\(
J(x)=\tfrac12\bigl(x+x^{-1}\bigr)-1, \, x>0.
\)
Then the following properties hold:
\begin{enumerate}
  \item[(i)] \(J\) is reciprocal and normalized: \(J(x)=J(x^{-1})\) for all \(x>0\) and \(J(1)=0\).
  \item[(ii)] \(J\) satisfies the composition law on \(\Rp\) (Definition~\ref{def:rp-law}).
  \item[(iii)] \(J\) has unit log-curvature: \(\kappa(J)=1\).
\end{enumerate}
\end{Lemma}
\begin{proof}
(i) Reciprocity follows directly from the definition of $J(x)$. Also \(J(1)=\tfrac12(1+1)-1=0\).

\noindent(ii) Let \(H(t)=J(\e^t)+1, \; t\in\R.\)
Let us first compute \(H\)
\begin{align*}
H(t)
&= \left(\frac12\bigl(\e^t+\e^{-t}\bigr)-1\right)+1
 = \frac12\bigl(\e^t+\e^{-t}\bigr)
 = \cosh(t).
\end{align*}
Hence \(H(t)=\cosh(t)\) for all \(t\in\R\).

The function \(\cosh\) satisfies the d'Alembert equation  
\begin{equation}\label{eq:H-dalembert}
H(t+u)+H(t-u)=2\,H(t)\,H(u)\qquad\text{for all }t,u\in\R.
\end{equation}
Let \(x,y>0\) and
\[
t=\ln x,\qquad u=\ln y.
\]
Then \(\e^t=x\), \(\e^u=y\), and consequently
\[
\e^{t+u}=xy,\qquad \e^{t-u}=\frac{x}{y}.
\]
Using the definition of \(H\), we can rewrite \eqref{eq:H-dalembert} as
\begin{align*}
\bigl(J(\e^{t+u})+1\bigr)+\bigl(J(\e^{t-u})+1\bigr)
&=2\bigl(J(\e^t)+1\bigr)\bigl(J(\e^u)+1\bigr).
\end{align*}
Substituting \(\e^{t+u}=xy\), \(\e^{t-u}=x/y\), \(\e^t=x\), and \(\e^u=y\), we obtain
\[
\bigl(J(xy)+1\bigr)+\bigl(J\Big(\dfrac xy\Big)+1\bigr)
=2\bigl(J(x)+1\bigr)\bigl(J(y)+1\bigr).
\]
Therefore, $J$ satisfies the composition law on
\(\Rp\) given by Definition~\ref{def:rp-law}:
\[
J(xy)+J\Bigl(\frac{x}{y}\Bigr)
=2\,J(x)\,J(y)+2\,J(x)+2\,J(y),
\qquad x,y>0.
\]

\noindent (iii) Using \(J(\e^t)=\cosh(t)-1\) and the Taylor expansion
\(\cosh(t)=1+t^2/2+o(t^2)\) as \(t\to 0\), we have
\[
\lim_{t\to 0}\frac{2J(\e^t)}{t^2}
=\lim_{t\to 0}\frac{2(\cosh(t)-1)}{t^2}=1,
\]
so \(\kappa(J)=1\).
\end{proof}

The next theorem shows that a calibrated d’Alembert solution is forced onto a single cosine/hyperbolic-cosine branch, with the calibration determining the parameter.

\begin{Theorem}\label{thm:cosh-unique}
Let \(H:\R\to\R\) satisfy the d'Alembert equation (\ref{dal}).
Assume the following limit exists:
\begin{equation}\label{kal}
\kappa_H:=\lim_{t\to 0}\frac{2\,(H(t)-1)}{t^2}\in\R.
\end{equation}
Then:
\begin{enumerate}
  \item If \(\kappa_H>0\), then \(H(t)=\cosh(\sqrt{\kappa_H}\,t)\) for all \(t\in\R\).
  \item If \(\kappa_H<0\), then \(H(t)=\cos(\sqrt{-\kappa_H}\,t)\) for all \(t\in\R\).
  \item If \(\kappa_H=0\), then \(H(t)=1\) for all \(t\in\R\).
\end{enumerate}
In particular, if \(\kappa_H=1\), then \(H(t)=\cosh(t)\) for all \(t\in\R\).
\end{Theorem}

\begin{proof} 
By Lemma~\ref{lem:dalembert-continuity}, the existence of
\[
\lim_{t\to 0}\frac{2\,(H(t)-1)}{t^2}\in\R
\]
implies that \(H\) is continuous on \(\R\).  
Hence we will apply the classical classification of continuous real-valued
solutions of the d'Alembert equation \eqref{dal} (see, for example, \cite{Aczel}).
\[
H(t)\equiv 1,\quad
H(t)=\cos(kt)\quad\text{or}\quad H(t)=\cosh(kt),\qquad k\in\R.
\]

If \(H(t)\equiv 1\), then \(H(t)-1\equiv 0\), so \(\kappa_H=0\), and the conclusion in (3) holds.

\medskip
If \(H(t)=\cosh(kt)\) for some \(k\in\R\).
Using the Taylor expansion \(\cosh(t)=1+t^2/2+o(t^2)\) as \(t\to 0\), we have 
\[
\cosh(kt)-1=\frac{(kt)^2}{2}+o(t^2),\qquad t\to 0.
\]
Therefore,
\[
\kappa_H
=\lim_{t\to 0}\frac{2(\cosh(kt)-1)}{t^2}
=\lim_{t\to 0}\frac{2\left(\frac{k^2t^2}{2}+o(t^2)\right)}{t^2}
=k^2\ge 0.
\]
If \(\kappa_H>0\), then \(k\neq 0\) and \(|k|=\sqrt{\kappa_H}\), hence
\[
H(t)=\cosh(kt)=\cosh(\sqrt{\kappa_H}\,t),\qquad t\in\R,
\]
since \(\cosh\) is even. If \(k=0\), then \(H\equiv 1\) and \(\kappa_H=0\),
which is already covered by the first case.

\medskip
Similarly, if \(H(t)=\cos(kt)\) for some \(k\in\R\), the Taylor expansion \(\cos(kt)-1=-k^2t^2/2+o(t^2)\) yields \(\kappa_H=-k^2\), and the claim in (2) follows.
\end{proof}

\begin{Corollary}\label{cor:cost-unique}
Let \(F:\Rp\to\R\). Assume \(F\) satisfies the
composition law on \(\mathbb{R}_{>0}\) and has unit log-curvature \(\kappa(F)=1\). Then
\[
F(x)=\frac{x+x^{-1}}{2}-1
\qquad\text{for all }x>0.
\]
\end{Corollary}
\begin{proof}
Let \(H(t)=F(\e^t)+1\). By Lemma~\ref{lem:equiv-rp-dalembert}, \(H\) satisfies d'Alembert
equation (\ref{dal}), and
moreover,
\[
\kappa_H=\lim_{t\to 0}\frac{2(H(t)-1)}{t^2}=\lim_{t\to 0}\frac{2F(\e^t)}{t^2}=\kappa(F)=1.
\]
From Theorem~\ref{thm:cosh-unique} then \(H(t)=\cosh(t)\), hence
\(
F(\e^t)=\cosh(t)-1=J(\e^t)
\). For \(x>0\) and \(x=\e^{\ln x}\), we have \(F(x)=J(x)\).
\end{proof}


For completeness we record an alternative route to the classification.
The existence of the curvature limit implies sufficient smoothness to differentiate \eqref{dal} and obtain a linear ODE for \(H\).
These lemmas provide an independent way to the classification, although Theorem~\ref{thm:cosh-unique} as stated follows directly from the classical classification of continuous solutions
of the d'Alembert equation.


\begin{Lemma}\label{lem:dalembert-curvature-ode}
Let \(H:\R\to\R\) satisfy the d'Alembert equation (\ref{dal}). Suppose the following limit exists
\begin{equation*}
\kappa_H:=\lim_{t\to 0}\frac{2\,(H(t)-1)}{t^2}\in\R.
\end{equation*}Then \(H\in C^2(\R)\) and
\[
H''(t)=\kappa_H\,H(t)\qquad\text{for all }t\in\R.
\]
\end{Lemma}




{  \begin{proof}
From the d'Alembert equation (\ref{dal}),
we obtain
\[
\frac{H(t+h)-2H(t)+H(t-h)}{h^2}
=
2H(t)\frac{H(h)-1}{h^2}.
\]
By assumption,
\[
\frac{2(H(h)-1)}{h^2}\to \kappa_H \qquad (h\to0),
\]
hence the left-hand side converges to $\kappa_H H(t)$.
Therefore $H\in C^2(\R)$ and $H''(t)=\kappa_H H(t)$.
\end{proof}}

\begin{Lemma} \label{lem:dalembert-to-ode}
Let \(H\in C^2(\R)\) satisfy the d'Alembert equation (\ref{dal}). Then for all \(t\in\R\),
\[
H''(t)=H''(0)\,H(t).
\]
\end{Lemma}
{ \begin{proof}

Taking the second derivative of (\ref{dal}) with respect to \(u\) for fixed $t$, we get
\[
H''(t+u)+H''(t-u)-2H(t)H''(u)=0.
\]
Therefore, evaluating at $u=0$, we have
\[
2H''(t)-2H(t)H''(0)=0,
\]
so \(H''(t)=H''(0)H(t)\) for all \(t\).
\end{proof}}

\begin{Lemma}\label{lem:even-deriv0}
Let \(H\in C^1(\R)\) be even. Then \(H'(0)=0\).
\end{Lemma}

{  \begin{proof} This follows immediately from $H(t)=H(-t)$ by differentiating at $t=0$.\end{proof}}



{ \begin{Lemma}\label{lem:ode-zero}
Let $\kappa\in\R\setminus\{0\}$ and let $f\in C^2(\R)$ satisfy
\(
f''(t)=\kappa f(t)\, \text{for all }t\in\R,
\)
with $f(0)=0$ and $f'(0)=0$. Then \(f(t)=0\) for all \(t\).
\end{Lemma}

\begin{proof}
The equation $f''(t)=\kappa f(t)$ is a linear ODE with continuous coefficients, hence the
initial value problem with data $(f(0),f'(0))=(0,0)$ has a unique solution.
Since $f\equiv 0$ is a solution, uniqueness implies $f\equiv 0$.
\end{proof}}



In terms of logarithmic coordinates \(H(t)=F(e^t)+1\), the rigidity problem can be summarized as follows:
once the d’Alembert equation \eqref{dal} holds and the quadratic calibration \eqref{kal} exists, the solution is forced into one of the standard cosine/hyperbolic-cosine families, with the parameter fixed by \(\kappa_H\).
In our case, the normalization \(F(1)=0\) forces \(H(0)=1\), and the unit calibration in Theorem~\ref{thm:main} forces \(\kappa_H=1\), hence \(H(t)=\cosh(t)\) and \(F=J\).

\smallskip  Theorem~\ref{thm:main} is intentionally minimal: each assumption plays a distinct role. The following propositions show what fails when individual assumptions are removed.

\begin{Proposition}\label{prop:reciprocity-forced}
Let \(F:\Rp\to\R\) satisfy \eqref{comp} and \(F(1)=0\). Then \(F(x)=F(x^{-1})\) for all
\(x>0\).
\end{Proposition}
\begin{proof}
Plug \(x=1\) into Definition~\ref{def:rp-law} to obtain
\[
F(y)+F\Big(\dfrac 1 y\Big)=2F(1)F(y)+2F(1)+2F(y)=2F(y),
\]
hence \(F\Big(\dfrac 1 y\Big)=F(y)\) for all \(y>0\).
\end{proof}

\begin{Proposition}\label{prop:family-via-W}
Assume there is a function \(W:\Rp\to(0,\infty)\) satisfying 
\[
W(xy)=W(x)W(y)\qquad x,y>0.
\]
Define \(F_W:\Rp\to\R\) by
\[
F_W(x):=\frac{W(x)+W(x)^{-1}}{2}-1.
\]
Then \(F_W(1)=0\), and \(F_W\) satisfies the \(\Rp\) composition law
(Definition~\ref{def:rp-law}). If, in addition, \(W\) is continuous, then there exists \(\lambda\in\R\)
with \(W(x)=x^\lambda\) for all \(x>0\), hence
\[
F_W(x)=\cosh(\lambda\ln x)-1.
\]
\end{Proposition}
\begin{proof}
First, \(W(1)=W(1\cdot 1)=W(1)^2\) and \(W(1)>0\) force \(W(1)=1\), so \(F_W(1)=\tfrac12(1+1)-1=0\).
Also,
\[
W(x)W(x^{-1})=W(xx^{-1})=W(1)=1,
\]
so \(W(x^{-1})=W(x)^{-1}\) and hence \(F_W(x)=F_W(x^{-1})\).

Define \(H(t):=F_W(\e^t)+1=\tfrac12(W(\e^t)+W(\e^t)^{-1})\). Let \(t,u\in\R\). Then
\begin{align*}
H(t+u)+H(t-u)
&=\frac{W(\e^{t+u})+W(\e^{t+u})^{-1}+W(\e^{t-u})+W(\e^{t-u})^{-1}}{2}\\
&=\frac{(W(\e^t)+W(\e^t)^{-1})(W(\e^u)+W(\e^u)^{-1})}{2}
=2H(t)H(u),
\end{align*}
so \(H\) satisfies d'Alembert. By Lemma~\ref{lem:equiv-rp-dalembert}, \(F_W\) satisfies the \(\Rp\)
composition law.

Finally, if \(W\) is continuous, define \(w(t):=\ln W(\e^t)\). Then \(w:\R\to\R\) is additive:
using \(W(\e^{t+u})=W(\e^t)W(\e^u)\), we get \(w(t+u)=w(t)+w(u)\). Continuity of \(W\) implies
continuity of \(w\), hence \(w(t)=\lambda t\) for some \(\lambda\in\R\). Therefore
\(
W(\e^t)=\e^{\lambda t}
\),
i.e. \(W(x)=x^\lambda\) for \(x>0\), and
\(
F_W(\e^t)=\cosh(\lambda t)-1
\),
equivalently \(F_W(x)=\cosh(\lambda\ln x)-1\).
\end{proof}

\begin{Corollary}\label{cor:no-calibration-family}
For any \(\lambda>0\), define
\[
F_\lambda(x):=\cosh(\lambda\ln x)-1 \qquad (x>0).
\]
Then \(F_\lambda\) is continuous on \(\Rp\), satisfies \(F_\lambda(1)=0\),
satisfies the composition law \eqref{comp}, and has \(\kappa(F_\lambda)=\lambda^2\).
In particular, the unit calibration \(\kappa(F)=1\) fixes \(\lambda=1\).
\end{Corollary}

\begin{proof}
Take \(W(x)=x^\lambda\) in the previous proposition, so \(F_\lambda=F_W\) and \eqref{comp} holds.
Moreover, since \(\cosh(\lambda t)-1=\frac{\lambda^2 t^2}{2}+o(t^2)\) as \(t\to 0\), we get
\[
\kappa(F_\lambda)
=\lim_{t\to 0}\frac{2(\cosh(\lambda t)-1)}{t^2}
=\lambda^2.
\]
\end{proof}

The following proposition shows that normalization and unit calibration alone 
do not imply the composition law on \(\mathbb{R}_{>0}\).

\begin{Proposition}\label{prop:no-law-counterexample}
Let us define \(F(x):=\tfrac12(\ln x)^2\) on \(\Rp\). Then \(F\) is
continuous, satisfies \(F(1)=0\), and has \(\kappa(F)=1\), but it does
\emph{not} satisfy the  composition law on \(\mathbb{R}_{>0}\).
\end{Proposition}
\begin{proof}
Continuity and \(F(1)=0\) are immediate. Moreover,
\(F(\e^t)=t^2/2\), so
\[
\kappa(F)=\lim_{t\to 0}\frac{2(t^2/2)}{t^2}=1.
\]
Define \(H(t):=F(\e^t)+1=1+t^2/2\). Then for \(t,u\in\R\),
\[
H(t+u)+H(t-u)=2+\frac{(t+u)^2+(t-u)^2}{2}=2+t^2+u^2,
\]
while
\[
2H(t)H(u)
=2\left(1+\frac{t^2}{2}\right)\left(1+\frac{u^2}{2}\right)
=2+t^2+u^2+\frac{t^2u^2}{2}.
\]
The last two expressions differ when \(tu\neq 0\), hence \(H\) does not satisfy
d'Alembert's functional equation. By Lemma~\ref{lem:equiv-rp-dalembert},
\(F\) does not satisfy the composition law on \(\mathbb{R}_{>0}\).

\end{proof}

The following proposition shows that, without a regularity assumption,
the composition law on \(\mathbb{R}_{>0}\) (\ref{comp}) admits pathological solutions.

\begin{Proposition}\label{prop:pathological}
There exist functions \(F:\mathbb{R}_{>0}\to\mathbb{R}\) satisfying \(F(1)=0\) and the
composition law on \(\mathbb{R}_{>0}\)~\eqref{comp} that are not measurable (and hence not continuous).
\end{Proposition}

\begin{proof}
It is a  consequence of the existence of a Hamel basis of \(\mathbb{R}\) over \(\mathbb{Q}\)
that there exists an additive function \(a:\mathbb{R}\to\mathbb{R}\),
\(a(t+u)=a(t)+a(u)\), which is not measurable (and hence not continuous).

Let us define
\[
H(t):=\cosh(a(t)), \qquad t\in\mathbb{R}.
\]

We claim that \(H\) is not measurable. Suppose for contradiction that \(H\) is measurable.
Since \(\operatorname{arcosh}:[1,\infty)\to[0,\infty)\) is continuous and \(\cosh(a(t))\ge 1\), the composition
\[
|a(t)|=\operatorname{arcosh}(H(t))
\]
is measurable. For each \(n\in\mathbb{N}\), define the measurable sets
\[
E_n:=\{t\in\R:\ |a(t)|\le n\}.
\]
Then \(\bigcup_{n=1}^\infty E_n=\R\), so at least one \(E_n\) has positive Lebesgue measure.
On that set, \(a\) is bounded. It is a classical fact that an additive function bounded on a set of positive measure must be linear, hence continuous, contradicting the choice of \(a\).
Therefore \(H\) is not measurable.

Using the additivity of \(a\) and the identity
\(\cosh(x+y)+\cosh(x-y)=2\cosh(x)\cosh(y)\), we compute for all \(t,u\in\mathbb{R}\),
\begin{align*}
H(t+u)+H(t-u)
&=\cosh(a(t)+a(u))+\cosh(a(t)-a(u))\\
&=2\cosh(a(t))\cosh(a(u))
=2H(t)H(u).
\end{align*}
Thus \(H\) satisfies d'Alembert's functional equation and \(H(0)=\cosh(a(0))=1\).

Finally, we define
\[
F(x):=H(\ln x)-1, \qquad x>0.
\]
Then \(F(1)=0\), and by Lemma~\ref{lem:equiv-rp-dalembert}, \(F\) satisfies the composition law
on \(\mathbb{R}_{>0}\)~\eqref{comp}.
If \(F\) were measurable, then \(H(t)=F(e^t)+1\) would be measurable as a composition
of measurable functions, contradicting the non-measurability of \(H\).
Hence \(F\) is not measurable.
\end{proof}

{ \section{Consistency estimate under bounded d’Alembert defect}\label{sec:stability}}

In this section we give a local consistency estimate.
If the d’Alembert equation holds with a uniform defect on a compact set and the function is sufficiently smooth, then the solution is close to the hyperbolic cosine case.

\begin{Definition}\label{def:defect}
For \(H:\R\to\R\), we define the {\rm d'Alembert defect}
\begin{equation}\label{defect}
\Delta_H(t,u):=H(t+u)+H(t-u)-2H(t)H(u).
\end{equation}
\end{Definition}

\begin{Theorem}\label{thm:stability}
Fix \(T>0\). Let \(H\in C^3([-T,T])\) be even with \(H(0)=1\), and set \(a:=H''(0)\).
Assume \(a>0\).  Let
\[
\varepsilon:=\sup_{|t|\le T,\ |u|\le T}\ |\Delta_H(t,u)|,\qquad
B:=\sup_{|t|\le T}\ |H(t)|,\qquad
K:=\sup_{|t|\le T}\ |H^{(3)}(t)|.
\]
Then for every \(h\) with \(0<h\le T\) and every \(t\) with \(|t|\le T-h\),
\[
\bigl|H(t)-\cosh(\sqrt{a}\,t)\bigr|
\le
\frac{\delta(h)}{a}\,\bigl(\cosh(\sqrt{a}\,|t|)-1\bigr),
\]
where
\[
\delta(h):=\frac{\varepsilon}{h^2}+\frac{(1+B)K}{3}\,h.
\]
\end{Theorem}
\begin{proof}
Fix \(0<h\le T\) and \(|t|\le T-h\).

Using the integral form of Taylor's theorem, we have
\[
H(t+h)=H(t)+hH'(t)+\frac{h^2}{2}H''(t)+\int_0^h\frac{(h-s)^2}{2}\,H^{(3)}(t+s)\,ds,
\]
\[
H(t-h)=H(t)-hH'(t)+\frac{h^2}{2}H''(t)-\int_0^h\frac{(h-s)^2}{2}\,H^{(3)}(t-s)\,ds.
\]
Adding the last two equations and bounding \(|H^{(3)}|\le K\) yields \begin{equation}\label{eq:taylor1}
\bigl|H(t+h)+H(t-h)-2H(t)-h^2H''(t)\bigr|\le \frac{K}{3}\,h^3.
\end{equation}
Since \(H\) is even, \(H'(0)=0\), and the integral form at \(0\) gives
\begin{equation}\label{eq:taylor2}
\bigl|H(h)-1-\frac{a}{2}h^2\bigr|\le \frac{K}{6}\,h^3.
\end{equation}
Now, by definition of $\Delta_H$ (\ref{defect}), we have for $(t,h)$ 
\[
H(t+h)+H(t-h)=2H(t)H(h)+\Delta_H(t,h).
\]
Subtract \(2H(t)+a h^2 H(t)\) from both sides, we obtain
\begin{align*}
h^2\bigl(H''(t)-aH(t)\bigr)
&=\bigl(H(t+h)+H(t-h)-2H(t)-h^2H''(t)\bigr)\\
&\quad+\Delta_H(t,h)
+2H(t)\bigl(H(h)-1-\frac{a}{2}h^2\bigr).
\end{align*}
Taking absolute values in the last equation and using \eqref{eq:taylor1}, \eqref{eq:taylor2}, \(|H(t)|\le B\), and
\(|\Delta_H(t,h)|\le\varepsilon\), we obtain
\[
h^2|H''(t)-aH(t)|
\le \frac{K}{3}h^3+\varepsilon+2B\cdot\frac{K}{6}h^3
\le \varepsilon+\frac{(1+B)K}{3}h^3.
\]
Dividing by \(h^2\) yields the uniform bound
\begin{equation}\label{eq:resid}
|H''(t)-aH(t)|\le \delta(h)
\qquad(|t|\le T-h).
\end{equation}

Let \(y(t):=\cosh(\sqrt{a}\,t)\), so \(y''=ay\), \(y(0)=1\), and since \(H\) is even, \(H'(0)=0=y'(0)\).
Define \(e(t):=H(t)-y(t)\). Then \(e\in C^2([-T+h,T-h])\), \(e(0)=e'(0)=0\), and
\[
e''(t)-a e(t)=H''(t)-aH(t),
\]
so by \eqref{eq:resid}, \(|e''(t)-a e(t)|\le\delta(h)\) for \(|t|\le T-h\).
For \(t\in[0,T-h]\), the equation for \(e''=ae+r\) with zero initial condition gives
\[
e(t)=\int_0^t \frac{1}{\sqrt{a}}\sinh(\sqrt{a}(t-s))\,r(s)\,ds,
\]
where \(r(s):=e''(s)-ae(s)\). Hence
\[
|e(t)|
\le \delta(h)\int_0^t \frac{1}{\sqrt{a}}\sinh(\sqrt{a}(t-s))\,ds
=\frac{\delta(h)}{a}\bigl(\cosh(\sqrt{a}\,t)-1\bigr).
\]
Since \(e\) is even as difference of even functions, this bound holds for negative \(t\), yielding
the inequality for all \(|t|\le T-h\).
\end{proof}


{ \begin{Remark}
The use of the Taylor expansion in the proof requires the assumption 
$H\in C^3([-T,T])$. 
The result should be interpreted as a consistency estimate 
under smoothness assumptions rather than a stability result in the 
classical sense.
\end{Remark}}

\begin{Corollary}\label{cor:stability-rp}
Let the assumptions of Theorem~\ref{thm:stability} hold and define
\[
F(x):=H(\ln x)-1,\qquad x\in
{\mathbb{R}_{>0}}.
\]
Then for every \(x\in\bigl(e^{-(T-h)},e^{T-h}\bigr)\),
\[
\bigl|F(x)-\bigl(\cosh(\sqrt{a}\,\ln x)-1\bigr)\bigr|
\;\le\;
\frac{\delta(h)}{a}\Bigl(\cosh\bigl(\sqrt{a}\,|\ln x|\bigr)-1\Bigr).
\]

In particular, for every \(S\in(0,T-h)\) the function \(F\) is uniformly close to
\(\cosh(\sqrt{a}\,\ln x)-1\) on the compact interval \([e^{-S},e^{S}]\).

If moreover \(a\) is close to \(1\) and \(\delta(h)\) is small, then \(F\) is uniformly
close to \(J(x)=\cosh(\ln x)-1\) on \([e^{-S},e^{S}]\).

\medskip

In the special case \(a=1\), the estimate simplifies to
\[
\bigl|F(x)-J(x)\bigr|\le\delta(h)\,J(x),
\qquad x\in\bigl(e^{-(T-h)},e^{T-h}\bigr),
\]
since \(\cosh(|\ln x|)-1=\cosh(\ln x)-1=J(x)\).
\end{Corollary}

\begin{proof}
Apply Theorem~\ref{thm:stability} with \(t=\ln x\). Since \(F(x)=H(t)-1\), we obtain
\[
|F(x)-(\cosh(\sqrt a\,t)-1)|
\le \frac{\delta(h)}{a}\,(\cosh(\sqrt a\,|t|)-1),
\]
which is the required inequality after substituting \(t=\ln x\).

The case \(a=1\) follows from \(\cosh(|\ln x|)-1=\cosh(\ln x)-1=J(x)\).
\end{proof}


\section{Properties of canonical reciprocal cost}\label{sec:properties}

Having shown that the canonical reciprocal cost
\[
J(x)=\frac12(x+x^{-1})-1
\]
is uniquely determined by the composition law on \(\Rp\) and the unit log-curvature calibration, we now give some properties of \(J\) that make it interpretable and useful.

\subsection{Arithmetic and geometric means}
The function \(J(x)\) can be written as
\[
J(x)
=\mathrm{AM}\Big(x,\dfrac 1x\Big)-\mathrm{GM}\Big(x,\dfrac 1x\Big),
\]
since \(\mathrm{GM}(x,1/x)=1\) and \(\mathrm{AM}(x,1/x)=\tfrac12(x+x^{-1})\).
In particular \(J(x)\ge 0\) for all \(x>0\), with equality if and only if \(x=1\).

 
\subsection{Bregman divergence.}
The logarithmic expression of the canonical reciprocal cost is
\[
G(t)=J(\e^t)=\cosh t-1.
\]
Let \(\Phi:\R\to\R\) be defined by \(\Phi(t)=\cosh t\).
Then \(G\) coincides with the Bregman divergence generated by \(\Phi\) at the point \(0\), namely
\[
G(t)=D_\Phi(t,0)
=\Phi(t)-\Phi(0)-\Phi'(0)(t-0),
\]
since \(\Phi(0)=1\) and \(\Phi'(0)=\sinh 0=0\).

Since \(\Phi\in C^2(\R)\) and
\[
\Phi''(t)=\cosh t>0 \qquad \text{for all } t\in\R,
\]
the function \(\Phi\) is strictly convex. Consequently,
\[
D_\Phi(t,0)>0 \quad \text{for } t\neq 0,
\qquad
D_\Phi(0,0)=0,
\]
and \(0\) is the unique global minimizer of \(\Phi\).

Thus, in logarithmic coordinates the canonical cost \(J\) corresponds to the Bregman divergence
\(G(t)=D_\Phi(t,0)\). 
The strict convexity of $\Phi$ implies nonnegativity and a unique minimum at
$t=0$, while the Taylor expansion of $\cosh$ at $0$ gives the required quadratic
behavior.



\subsection{Metric associated with \texorpdfstring{$J$}{J}.}
Starting from the canonical reciprocal cost function
\[
J(x)=\dfrac12\Big(x+\frac1x\Big)-1,
\]
we introduce logarithmic coordinates \(t=\ln x\) and set
\[
G(t)=J(\e^t)=\cosh t-1,
\qquad
\Phi(t)=G(t)+1=\cosh t.
\]
The function \(\Phi\) is smooth and strictly convex on \(\R\), since
\[
\Phi''(t)=\cosh t>0 \qquad \text{for all } t\in\R.
\]
The Hessian metric induced by \(\Phi\) on \(\R\) is
\[
ds^2=\Phi''(t)\,dt^2=\cosh t\,dt^2.
\]
In \(x\)-coordinates, we get
\[
ds^2=\cosh(\ln x)\,\frac{dx^2}{x^2}
=\frac{x^2+1}{2x^3}\,dx^2.
\]
The corresponding Riemannian distance is 
\[
d_J(x,y)
=\left|\int_{\ln x}^{\ln y}\sqrt{\cosh u}\,du\right|
=\left|\int_x^y \sqrt{\frac{\xi^2+1}{2\xi^3}}\,d\xi\right|.
\]

Since \(\cosh\) is an even function, the distance is reciprocally symmetric, i.e.
\[
d_J(x,y)=d_J\Big(\frac1x,\frac1y\Big).
\]

Near \(x=1\) (that is, \(t=0\)), one has \(\cosh t=1+O(t^2)\), and hence
\(d_J(x,y)\) is locally equivalent to the logarithmic distance
\(\lvert \ln y-\ln x\rvert\).
More precisely, there exist constants \(c,C>0\) and a neighborhood \(U\) of \(1\)
such that
\[
c\,\lvert \ln y-\ln x\rvert
\le d_J(x,y)
\le C\,\lvert \ln y-\ln x\rvert,
\qquad x,y\in U.
\]

Globally, the growth of \(d_J\) is qualitatively different from the logarithmic distance because the metric weight \(\sqrt{\cosh u}\) grows exponentially as \(|u|\to\infty\).

As \(|t|\to\infty\), one has \(\cosh t\sim \tfrac12 e^{|t|}\), and therefore
\[
d_J(1,R)\sim \sqrt{2}\,R^{1/2}
\qquad \text{as } R\to\infty,
\]
reflecting the exponential growth of \(\Phi''(t)\) for large \(|t|\).
\subsection{Chebyshev structure of \texorpdfstring{$J$}{J}}

It is known that normalized solutions of d’Alembert equation
\[
H(m+n)+H(m-n)=2H(m)H(n), \qquad m,n\in\mathbb Z,
\]
can be expressed in terms of Chebyshev polynomials (see, for example, \cite{Davison}).
In particular, such solutions satisfy 

\begin{equation}\label{rec}
H_{n+1}=2H_1H_n-H_{n-1}
\end{equation}
and can be written in the form \(H_n=T_n(H_1)\), where \(T_n\) denotes the
Chebyshev polynomials.

We now show how this discrete Chebyshev structure is realized by the canonical reciprocal cost.
Let \(H(t)=J(e^t)+1=\cosh t\), which satisfies the continuous d’Alembert equation \eqref{dal}.
Fix \(t=\ln x\) and define the discrete sequence \(H_n:=H(nt)\) for \(n\in\mathbb{Z}\). Then \((H_n)\) satisfies the recursion \eqref{rec}, and
\[
J(x^n)+1=\cosh(n\ln x)
= T_n(\cosh(\ln x))
= T_n(J(x)+1),
\]
that is,
\[
J(x^n)=T_n(J(x)+1)-1 ,\qquad n\in\mathbb Z.
\]

Since \(H_1=\cosh(\ln x)\ge 1\), the Chebyshev identity \(T_n(\cosh s)=\cosh(ns)\) applies (with \(s=\ln x\)), yielding \(H_n=\cosh(n\ln x)\).
In particular, for \(x\neq 1\) we have \(H_1>1\), so the oscillatory cosine branch (which takes values in \([-1,1]\)) is not compatible with the canonical cost $J$.

Thus, the Chebyshev recursion associated with the discrete d’Alembert equation
is realized explicitly by the functional \(J\).




\subsection{Energy interpretation}
Although \(J\) is characterized here only by a functional equation, it is natural to view it heuristically as an energy-like penalty for multiplicative imbalance.
Indeed, \(J\ge 0\) with a unique minimizer at equilibrium \(x=1\), and in logarithmic coordinates \(x=e^t\) one has
\[
J(e^t)=\cosh t-1=\frac{t^2}{2}+O(t^4)\qquad (t\to 0),
\]
which matches the standard second-order behavior near a stable equilibrium in many variational models.
This interpretation should be read as an analogy: \(J\) is not claimed to be a Hamiltonian energy of a specific dynamical system.

\section{Conclusion}

 This paper studies a rigidity problem for functions 
\(F:\mathbb{R}_{>0}\to\mathbb{R}\) satisfying a polynomial composition law.
The normalization \(F(1)=0\) follows from the unit curvature calibration.
We determine which structural assumptions force \(F\) to have a unique functional form.

Under the composition law \eqref{comp} on \(\R_{>0}\)
and the unit log-curvature calibration \eqref{behavior0},
\(F\) is uniquely determined.
The resulting function is the canonical reciprocal cost
\[
J(x)=\frac12(x+x^{-1})-1,
\]
equivalently \(J(e^t)=\cosh(t)-1\) in logarithmic coordinates.

\smallskip

The composition law becomes d’Alembert’s functional equation \eqref{dal} on \(\R\),
and the calibration becomes a curvature condition at \(t=0\).
The calibration plays two distinct roles: it provides minimal regularity
so that the classical classification of solutions of d’Alembert’s equation applies,
and it fixes the remaining scale parameter that otherwise yields the family
\(H(t)=\cosh(\lambda t)\).
Consequently, the calibrated solution is forced onto the single branch
\(H(t)=\cosh(t)\), hence \(F=J\).

\smallskip

Section~\ref{sec:results} also establishes the minimality of the assumptions.
The composition law together with \(F(1)=0\) implies reciprocity \(F(x)=F(x^{-1})\).
Without fixing the calibration one obtains a one-parameter family of continuous
solutions \(F_\lambda(x)=\cosh(\lambda\ln x)-1\).
On the other hand, the normalization \(F(1)=0\) together with the calibration
does not imply the composition law \eqref{comp}
(see Proposition~\ref{prop:no-law-counterexample}).
Finally, even under the normalization \(F(1)=0\),
the composition law \eqref{comp} admits non-measurable solutions
(see Proposition~\ref{prop:pathological}).
Hence some additional regularity assumption is necessary.

\smallskip

Several directions remain for further study.
An open question is whether the composition law \eqref{comp} can be derived
from weaker assumptions, rather than being imposed explicitly.
Another direction is to consider alternative calibration conditions and to classify
other polynomial-type relations between \(F(xy)+F(x/y)\) and the pair \((F(x),F(y))\).

It would also be of interest to extend the analysis to other multiplicative structures,
such as positive definite matrices \cite{Bhatia} or more general structures.

\end{document}